\documentclass[11pt]{amsart}

\usepackage[foot]{amsaddr}

\usepackage{amssymb,amsfonts,amsmath,graphicx,verbatim,amsthm,calc} 
\usepackage{color}
\usepackage{xcolor}
\usepackage{newfloat}
\usepackage{hyperref}
\usepackage{faktor}

\usepackage{enumitem}
\usepackage[hypcap=false]{caption}

\usepackage{tikz}
\hypersetup{hidelinks, 
    colorlinks=true,       
    linkcolor=violet,          
    citecolor=green,        
    filecolor=magenta,      
    urlcolor=cyan           
}

\colorlet{genial}{black} 
\colorlet{genialsol}{black}

\makeatletter

\newtheoremstyle{genialnumbox}
{7pt}
{7pt}
{\normalfont}
{}
{\small\bf\sffamily\color{genial}}
{\;}
{0.25em}
{%
{\small\sffamily\color{genial}\thmname{#1}}%
{\nobreakspace\thmnumber{\@ifnotempty{#1}{}\@upn{#2}}}
\thmnote{{\nobreakspace\the\thm@notefont\sffamily\bfseries\color{black}\nobreakspace(#3)}} 
}

\newtheoremstyle{blacknumex}
{7pt}
{7pt}
{\normalfont}
{} 
{\small\bf\sffamily}
{\;}
{0.25em}
{%
{\small\sffamily\color{genial}\thmname{#1}}%
{\nobreakspace\thmnumber{\@ifnotempty{#1}{}\@upn{#2}}}
\thmnote{{\nobreakspace\the\thm@notefont\sffamily\bfseries\color{black}\nobreakspace(#3)}} 
}

\newtheoremstyle{blacknumbox} 
{7pt}
{7pt}
{\normalfont}
{}
{\small\bf\sffamily}
{\;}
{0.25em}
{%
{\small\sffamily\color{genial}\thmname{#1}}%
{\nobreakspace\thmnumber{\@ifnotempty{#1}{}\@upn{#2}}}
\thmnote{{\nobreakspace\the\thm@notefont\sffamily\bfseries\color{black}\nobreakspace(#3)}} 
}

\newtheoremstyle{genialnum}
{7pt}
{7pt}
{\normalfont}
{}
{\small\bf\sffamily\color{genial}}
{\;}
{0.25em}
{%
{\small\sffamily\color{genial}\thmname{#1}}%
{\nobreakspace\thmnumber{\@ifnotempty{#1}{}\@upn{#2}}}
\thmnote{{\nobreakspace\the\thm@notefont\sffamily\bfseries\color{black}\nobreakspace(#3)}} 
}

\makeatother

\RequirePackage[framemethod=default]{mdframed} 

\newmdenv[skipabove=7pt,
skipbelow=7pt,
rightline=false,
leftline=false,
topline=false,
bottomline=false,
backgroundcolor=black!5,
linecolor=genial,
innerleftmargin=5pt,
innerrightmargin=5pt,
innertopmargin=7pt,
innerbottommargin=7pt,
leftmargin=0cm,
rightmargin=0cm]{tBox}

\newmdenv[skipabove=7pt,
skipbelow=7pt,
rightline=false,
leftline=false,
topline=false,
bottomline=false,
backgroundcolor=genial!10,
linecolor=genial,
innerleftmargin=5pt,
innerrightmargin=5pt,
innertopmargin=5pt,
innerbottommargin=5pt,
leftmargin=0cm,
rightmargin=0cm,
linewidth=4pt]{eBox}	

\newmdenv[skipabove=7pt,
skipbelow=7pt,
rightline=false,
leftline=true,
topline=false,
bottomline=false,
linecolor=genial!50,
innerleftmargin=5pt,
innerrightmargin=5pt,
innertopmargin=5pt,
leftmargin=0cm,
rightmargin=0cm,
linewidth=4pt,
innerbottommargin=5pt]{dBox}	

\newmdenv[skipabove=7pt,
skipbelow=7pt,
rightline=false,
leftline=false,
topline=false,
bottomline=false,
linecolor=gray,
backgroundcolor=black!5,
innerleftmargin=5pt,
innerrightmargin=5pt,
innertopmargin=7pt,
leftmargin=0cm,
rightmargin=0cm,
linewidth=4pt,
innerbottommargin=5pt]{cBox}

\newmdenv[skipabove=7pt,
skipbelow=7pt,
rightline=false,
leftline=false,
topline=false,
bottomline=false,
linecolor=gray,
backgroundcolor=black!5,
innerleftmargin=5pt,
innerrightmargin=5pt,
innertopmargin=7pt,
leftmargin=0cm,
rightmargin=0cm,
linewidth=4pt,
innerbottommargin=5pt]{pBox}

\newmdenv[skipabove=7pt,
skipbelow=7pt,
rightline=false,
leftline=false,
topline=false,
bottomline=false,
linecolor=genialsol,
innerleftmargin=5pt,
innerrightmargin=5pt,
innertopmargin=0pt,
leftmargin=0cm,
rightmargin=0cm,
linewidth=4pt,
innerbottommargin=0pt]{solBox}	

\theoremstyle{genialnumbox}
\newtheorem{thm1}{Theorem}[section]
\newtheorem{ithm1}[thm1]{$\star$ THEOREM}
\newtheorem{ques1}[thm1]{Question}
\newtheorem{conj1}[thm1]{Conjecture}

\theoremstyle{blacknumex}
\newtheorem{exer}[thm1]{Exercise}
\newtheorem{exer*}[thm1]{$\ast$ Exercise}

\theoremstyle{blacknumbox}
\newtheorem{dfn1}[thm1]{Definition}

\theoremstyle{genialnum}
\newtheorem{cor1}[thm1]{Corollary}
\newtheorem{prop1}[thm1]{Proposition}
\newtheorem{lem1}[thm1]{Lemma}

\newtheorem{exm1}[thm1]{Example}

\newenvironment{thm}{\paragraph{ } \begin{tBox}\begin{thm1}}{\end{thm1}\end{tBox}}

\newenvironment{exe*}{\paragraph{ } \begin{eBox}\begin{exer*}}{\hfill{\color{genial}
\ensuremath{\diamond\diamond\diamond}}\end{exer*}\end{eBox}}
	
\newenvironment{exm}{\paragraph{ } \begin{exm1}}{\hfill{\tiny%
\ensuremath{\bigtriangleup\bigtriangledown\bigtriangleup}}\end{exm1}}
\newenvironment{cor}{\paragraph{ } \begin{cBox}\begin{cor1}}{\end{cor1}\end{cBox}}	
\newenvironment{ques}{\paragraph{ } \begin{cBox}\begin{ques1}}{\end{ques1}\end{cBox}}

\newenvironment{prop}{\paragraph{ } \begin{pBox}\begin{prop1}}{\end{prop1}\end{pBox}}	
\newenvironment{lem}{\paragraph{ } \begin{pBox}\begin{lem1}}{\end{lem1}\end{pBox}}

\newenvironment{lem*}[1]{\vspace{1ex}\noindent
{\bf Lemma* (#1).} [restatement]  \hspace{0.5em} \em }{ }

\newenvironment{thm*}[1]{\begin{cBox}
\vspace{1ex}\noindent 
{\bf Theorem* (#1).} [restatement]  \hspace{0.5em} }{\end{cBox}}

\theoremstyle{genialnum}

\newtheorem*{clm*}{Claim}

\newenvironment{sol}%
{\begin{solBox}
\par \noindent 
\scriptsize
{\bf Solution to ex:{\color{blue} \arabic{exer}}.}  {\color{red} \ \  :( } \\ }%
{\hfill {\color{blue} :) $\checkmark$} \end{solBox}}

\newcommand{\ENDEXER}{
{\expandafter\comment}
{\expandafter\endcomment}
}

\newtheorem{remark}[thm1]{Remark}

\makeatletter
\renewcommand{\@seccntformat}[1]{\llap{\textcolor{genial}{\csname the#1\endcsname}\hspace{1em}}}                    
\renewcommand{\section}{\@startsection{section}{1}{\z@}
{-4ex \@plus -1ex \@minus -.4ex}
{1ex \@plus.2ex }
{\normalfont\large\sffamily\bfseries}}
\renewcommand{\subsection}{\@startsection {subsection}{2}{\z@}
{-3ex \@plus -0.1ex \@minus -.4ex}
{0.5ex \@plus.2ex }
{\normalfont\sffamily\bfseries}}
\renewcommand{\subsubsection}{\@startsection {subsubsection}{3}{\z@}
{-2ex \@plus -0.1ex \@minus -.2ex}
{.2ex \@plus.2ex }
{\normalfont\small\sffamily\bfseries}}                        
\renewcommand\paragraph{\@startsection{paragraph}{4}{\z@}
{-2ex \@plus-.2ex \@minus .2ex}
{.1ex}
{\normalfont\small\sffamily\bfseries}}

\makeatother

\newcommand{\abs}[1]{\left\vert#1\right\vert}

\newcommand{\set}[1]{\left\{#1\right\}}

\newcommand{\Integer}{\mathbb{Z}}

\newcommand{\N}{\mathbb{N}}
\newcommand{\Z}{\Integer}

\newcommand{\R}{\mathbb{R}}

\newcommand{\eps}{\varepsilon}

\newcommand{\ie}{{\em i.e., }}
\newcommand{\eg}{{\em e.g., }}

\newcommand{\1}[1]{\mathbf{1}_{\set{ #1 } }}

\newcommand{\ind}[1]{\mathbf{1}_{ #1}}

\def\squareforqed{\hbox{\rlap{$\sqcap$}$\sqcup$}}
\def\qed{\ifmmode\squareforqed\else{\unskip\nobreak\hfil
\penalty50\hskip1em\null\nobreak\hfil\squareforqed
\parfillskip=0pt\finalhyphendemerits=0\endgraf}\fi}

\newcommand{\ignore}[1]{ }

\newcommand{\p}{\partial}

\newcommand{\vphi}{\varphi}

\newcommand{\AND}{\qquad \textrm{and} \qquad}

\newcommand{\define}[1]{\textbf{#1}}

\DeclareMathOperator{\supp}{supp}
\newcommand{\conv}{\mathsf{conv}}

\title[Finite Busemann boundary]{Groups with a finite Busemann boundary \\ are virtually cyclic}
\author{Corentin Bodart}
\address{CB: Mathematical Institute, University of Oxford, Oxford,  OX2 6GG, United Kingdom}
\email{corentin.bodart@maths.ox.ac.uk}
\author{Liran Ron-George}
\author{Ariel Yadin}
\address{LR, AY: Department of Mathematics, Ben-Gurion University of the Negev, Be'er Sheva, Israel}
\email{lirar@post.bgu.ac.il, yadina@bgu.ac.il}
\thanks{The first author was supported by the Swiss SNF grant P500PT-225420. Research supported by the Israel Science Foundation, grant no.\ 954/21. We would like to thank an anonymous referee for their helpful suggestions.}

\begin{document}

\begin{abstract}
This note is a continuation of the study of the relationship between the geometry of Cayley graphs and the size of its metric-functional boundary. We show that if there exists a Cayley graph with finitely many Busemann points then the underlying group is virtually cyclic. Together with previous works, this completes the full characterization of groups with finite metric-functional boundaries.
The main new notion introduced is that of {\em annihilators}. \hfill (MSC 2020: 20F65)
\end{abstract}

\maketitle

\section{Introduction}

Let $(X,d)$ be a metric space. Fix a basepoint $x_0 \in X$ and consider the so-called 
{\em Busemann functions} $b_x\colon X \to \R$ given by $b_x(y) = d(x,y) - d(x,x_0)$.  
If we identify $X$ with the set $\{ b_x \ : \ x \in X \}$, we can compactify $X$ by taking 
the closure of $\{ b_x \ : \ x \in X \}$ under the topology of pointwise convergence.
This results in a compact space denoted $\overline{ (X,d) }$.
The ``new'' part of this space, $\p (X,d) : = \overline{ (X,d) } \setminus \{ b_x \ : \ x \in X \}$ 
is called the \define{metric-functional boundary}, and its elements are called \define{metric-functionals}.
(If we were to use the topology of uniform convergence on bounded sets instead, this construction would result
in something known as the {\em horofunction boundary}
\cite{gromov1981hyperbolic}.  
Throughout this note we only consider countable discrete spaces $X$, so there is no distinction between 
horofunctions and metric-functionals.)

It has been proposed by A.\ Karlsson to use metric-functionals to study general metric spaces,
in analogy to the way that linear-functionals are used to study vector spaces.  This has been quite fruitful,
see \eg \cite{karlsson2001non, karlsson2021linear,  karlsson2021hahn, karlsson2024metric} and references therein.

Our focus is for the case where $X=G$ is a finitely generated group. In this case, the metrics of special interest are \define{left-invariant} metrics $d$, \ie metrics satisfying $d(zx,zy) = d(x,y)$ for all $x,y,z \in G$.
The canonical basepoint to choose for a group $G$ is, of course, the identity element $x_0=1$.

A group $G$ acts on $\{ f\colon G \to \R \ | \ f(1)=0 \}$ by $x.f(y) = f(x^{-1} y) - f(x^{-1})$.
For a left-invariant metric $d$ on a group $G$ one observes that $x.b_y= b_{xy}$,
which implies that $\p (G,d)$ is an invariant compact set.
This metric-functional boundary of a group has proved quite useful in studying the geometry of a group, 
especially in the context of {\em Cayley graphs}, see \eg 
\cite{BF20, BT24, develin, KB02, rieffel, RY23, RY24, TY16, Walsh, WW06} among many other works.

Thinking of metric-functionals as analogs of linear-functionals, 
a natural idea is to consider {\em annihilators} (defined below).  
As an application we completely characterize the case of a finite metric-functional boundary, 
showing that this property does not depend on the specific choice of Cayley graph, and is equivalent 
to having a finite-index cyclic subgroup.

\subsection{Notation}

All groups considered are assumed to be countable, unless explicitly stated otherwise.
The identity element in a group $G$ is denoted by $1 = 1_G$.


A metric $d$ on a group $G$ is called \define{left-invariant} if $d(zx,zy) = d(x,y)$ for all $x,y,z \in G$. For such a metric we write $|x| = |x|_d : = d(x,1)$. We will also assume that our metrics are \define{proper}, i.e., the ball $\{ y\in G: d(x,y)\le r\}$ is compact for all $x\in G$ and $r\ge 0$.


Given a finitely generated group $G$, and a finite, symmetric generating set $S$,
define the \define{Cayley graph} $\Gamma(G,S)$ as the graph whose vertex set is $G$
and edge set is $\{ \{x,xs\} \ : \ x \in G \ , \ s \in S \}$.
The graph metric on $\Gamma(G,S)$ is denoted by $d_S$. 
It is easily seen to be left-invariant. We write $|x|_S = |x|_{d_S}$.

A metric $d$ is called a \define{Cayley metric} if $d=d_S$ for some finite, symmetric generating set $S$. Groups equipped with their Cayley metrics are (coarsely) geodesic spaces.

%
%
%
%
%
%

It can be shown that if $\gamma = (\gamma_n)_n$ is an infinite geodesic in a Cayley graph $\Gamma(G,S)$ (see Section \ref{scn:geodesics} below for precise definitions),
then the limit $\gamma_\infty  : = \lim_n b_{\gamma_n} \in \overline{ (G,d_S) }$ exists and is 
an element of the metric-functional boundary $\p (G,d_S)$ (see \eg \cite[Lemma 3.2]{RY23} for a proof).  
Such a metric-functional, arising as the limit along a geodesic, is called a \define{Busemann point}.
The set of all Busemann points in $\p (G,d_S)$ is denoted by $\p_b(G,d_S)$.

For a group property $\mathcal{P}$, we say that a group $G$ is \define{virtually}-$\mathcal{P}$
if there is a finite index subgroup $[G:H] < \infty$ such that $H$ is $\mathcal{P}$.
For example, $G$ is virtually-$\Z$ if there exists a finite index subgroup $[G:H] < \infty$ such that
$H \cong \Z$ (where $\Z$ denotes the additive group of integers).

\subsection{Main results}

Our main result is the following.

\begin{thm} \label{thm:main}
Let $G$ be a finitely generated, infinite group. 
The following are equivalent.
\begin{enumerate}[topsep=-2mm, leftmargin=7mm, itemsep=2pt, label=(\arabic*)]
\item $G$ is virtually-$\Z$.
\item For any Cayley metric $d_S$ on $G$, the metric-functional boundary $\p (G,d_S)$ is finite. 
\item There exists a Cayley metric $d_S$ on $G$ such that the set of Busemann points $\p_b(G,d_S)$ is finite. 
\end{enumerate}
\end{thm}

The fact that all Cayley graphs of virtually-$\Z$ groups have only finitely many Busemann points was shown in \cite{TY16}. An alternative short proof was provided by Sam Shepperd, see \cite{RY23}. More strongly, the equivalence $(1)\Leftrightarrow(2)$, \ie every Cayley graph has a finite metric-functional boundary 
if and only if the group is virtually-$\Z$, was shown in \cite{RY23}. This note complements these previous works
proving that if there exists one Cayley graph with a finite set of Busemann points, then the group must be virtually-$\Z$
(and thus all Cayley graphs have a finite metric-functional boundary).

In order to prove Theorem \ref{thm:main}, we must first 
analyze the virtually-Abelian case.

\begin{thm} 
\label{thm:virt Abelian}
Let $d_S$ be a Cayley metric on a finitely generated, infinite, virtually Abelian group $G$. 
Then, $\p_b (G,d_S)$ is finite if and only if $G$ is virtually-$\Z$.
\end{thm}

The proof of Theorem \ref{thm:virt Abelian} is in Section \ref{scn:VA groups}.

In order to prove Theorem \ref{thm:main} we require the notion of {\em annihilators},
which will be introduced in the next section.
The main result is Lemma \ref{lem:N finite}, see below.
The proofs of Theorem \ref{thm:main} and of Lemma \ref{lem:N finite} are in Section \ref{scn:main results}.

\subsection{Annihilators}

Fix some left-invariant metric $d$ on a finitely generated group $G$.

In studying the structure of the set of Busemann points, a natural subgroup 
arises, as we now explain.

A group $G$ acts naturally on $\overline{ (G,d) }$ by $x.h(y) = h(x^{-1} y) - h(x^{-1})$.

A subset $L \subseteq \overline{ (G,d)} $ is called \define{invariant} if 
$x.L = \{ x.h \ : \ h \in L \} = L$ for all $x \in G$.

Note that $x.b_y = b_{xy}$, so $\{b_x:x\in G\}$ and its complement $\p (G,d)$ are invariant. Also, if $d_S$ is a Cayley metric, then for any infinite geodesic $\gamma = (\gamma_n)_n$
we have that $(x \gamma_n)_n$ is also a geodesic, so $\p_b (G,d_S)$ is an invariant set.

Given a subset $L \subseteq \p (G,d)$ we define the \define{annihilator} of $L$ to be
$$ N_L : = \bigl\{ x \in G \ : \ \forall \ h \in L \ , \ h(x) = 0 \bigr\} . $$
We write $N = N_{\p (G,d)}$, the annihilator of the whole metric-functional boundary.
Note that the smaller $L$ is, the larger $N_L$ becomes.

\begin{prop}
\label{prop:ann is sbgrp}
If $L \subseteq \p(G,d)$ is invariant, then $N_L$ is a subgroup.
\end{prop}
\begin{proof}
This follows directly from the identities $h(1)=0$, $h(xy) = h(x)+x^{-1}.h(y)$ and $h(x^{-1})=h(1)-x.h(x)=-x.h(x)$.
\end{proof}

We prove that the annihilators for two most prominent invariant subsets $L$ coincide.

\begin{thm} \label{thm:Busemann annihilator}
Let $d_S$ be a Cayley metric on a group $G$.

Then, $N_{\p_b(G,d_S)} = N$.
\end{thm}

The proof of Theorem \ref{thm:Busemann annihilator} is at the end of Section \ref{scn:geodesics}. This is non-trivial since $\partial_b(G,d_S)$ is often much smaller than $\partial(G,d_S)$. Examples are the lamplighter group \cite{lamplighter} and the Heisenberg group $H_3(\Z)$ \cite[Theorem A and B]{BT24}.

The annihilator provides a key step in the proof of Theorem \ref{thm:main}, stated in the following lemma and proved in Section \ref{scn:main results}.

\begin{lem}
\label{lem:N finite}
Let $d_S$ be a Cayley metric on a group $G$. 

If the set of Busemann points $\p_b(G,d_S)$ is finite then $N$ is finite.
\end{lem}

Finally, we have the following observation which could be of independent interest.
\begin{prop}
\label{prop:locally finite}
Let $d$ be a proper left-invariant metric on $G$ taking integer values. Then $N$ is a locally finite group; \ie any finitely generated subgroup of $N$ is finite.
\end{prop}

The proof of Proposition \ref{prop:locally finite} is in Section \ref{scn:locally_finite}.

Given a group $H$ we write $L(H) \in \{1,2,\ldots, \infty\}$ to be the supremum over the sizes of a locally finite subgroup of $H$.

\begin{lem}
Let $d$ be a proper, left-invariant, integer-valued metric on a group $G$.
Let $N = N_{\p(G,d)}$ be the annihilator subgroup.
Then,
$$ |N| \leq \inf_{H \leq G } [G:H] \cdot L(H) . $$ 
\end{lem}

\begin{proof}
Let $H \leq G$ be some subgroup. Notice that $[N:H\cap N] \leq [G:H]$, since whenever $(g_i)_{i\in I}\in N$ are in distinct cosets of $N\cap H$, then they are in distinct cosets of $H$.
Now, $H \cap N$ is a locally finite subgroup of $H$, implying that 
$$ |N| \leq [G:H] \cdot |H \cap N| \leq [G:H] \cdot L(H) . $$
\end{proof}

\begin{cor} \label{cor:torsion free}
If $H \leq G$ is a subgroup such that $H \cong \Z^d$,
then for any proper, left-invariant integer-valued metric 
the annihilator satisfies $|N| \leq [G:H]$.
\end{cor}

\begin{proof}
If $H \cong \Z^d$, then $H$ is torsion-free, and specifically $L(H) = 1$.
\end{proof}

Of course, this raises the question whether $N$ itself can be infinite, or if it must be finite (equivalently, finitely generated). We give an example at the end of Section \ref{scn:locally_finite}. However, the following question remains open:

\begin{ques}
Does there exists a Cayley metric $d_S$ for which $N=N_{\p(G,d_S)}$ is infinite?
\end{ques}

We conclude this section with two examples of Cayley metrics for which $N\ne \{1\}$.
\begin{exm}
    Consider $G=\Z\times\Z/n\Z$, generated by $S=\{(\pm1,0),(\pm1,\pm1)\}$. We have 
    \[ d_S\bigl((x,y),(x',y')\bigr) = \abs{x-x'} \]
    as soon as $\abs{x-x'}\ge n-1$. It follows that $d_S\bigl((0,y),g\bigr)=d_S\bigl(1,g\bigr)$ for all but finitely many $g\in\Z\times\Z/n\Z$, hence $N= \{(0,y):y\in\Z/n\Z\}$ by Lemma \ref{lem:indistinguishable}.
\end{exm}
\begin{center}
    \includegraphics[width=0.85\linewidth]{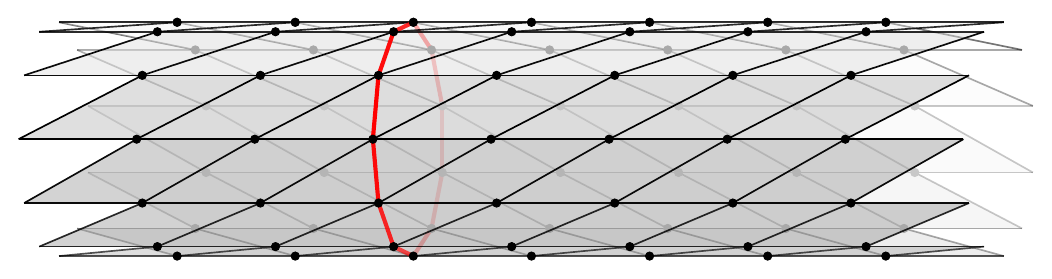}
    \captionof{figure}{The Cayley graph of $\Z\times\Z/n\Z$, with $N$ in red.}
    \label{fig:cylinder}
\end{center}

\begin{exm}
    Consider a group $G$, a finite symmetric generating set $S_1$ and a finite subgroup $F$. We take $S=FS_1F$ and prove that $N=N_{\partial(G,d_S)}\ge F$.
    
    For all $f\in F$ and $y\in G-F$, we have $d_S(f,y)=d_S(1,y)$. Indeed, if $y=s_1s_2\ldots s_\ell$ with $s_i\in S$, then $f^{-1}y=(f^{-1}s_1)s_2\ldots s_\ell$ and $f^{-1}s_1\in S$, proving that
    \[ d_S\bigl(f,y\bigr)=\abs{f^{-1}y}_S\le \abs{y}_S=d_S\bigl(1,y\bigr). \]
    The reverse inequality follows from the same argument applied to $f'=f^{-1}\in F$ and $y'=f^{-1}y\in G-F$. We conclude that $f\in N$ using Lemma \ref{lem:indistinguishable}.
\end{exm}





\section{Geodesics}

\label{scn:geodesics}

Let $d_S$ be a Cayley metric on a finitely generated group $G$.

A \define{finite geodesic} from $x$ to $y$ is a sequence $x=\gamma_0, \gamma_1 , \ldots, \gamma_\ell=y$
such that for all $0 \le i,j \le \ell$ we have $d_S(\gamma_i,\gamma_j) = \abs{i-j}$.
An \define{infinite geodesic} is a sequence $(\gamma_n)_{n=0}^\infty$ such that 
for all $ i,j\ge 0$ we have $d_S(\gamma_i,\gamma_j) = \abs{i-j}$.
The \define{length} of a geodesic $\gamma$ is $|\gamma|=\ell$ if $\gamma = (\gamma_0,\ldots, \gamma_\ell)$,
and $|\gamma|=\infty$ if $\gamma$ is an infinite geodesic.

Recall that $\gamma_\infty : = \lim_n b_{\gamma_n} \in \p_b (G,d_S)$.

For a metric space $(X,d)$ and $x,y \in X$ define the \define{segment}
between $x,y$ to be
$$ [x,y] : = \bigl\{ z \in X \ : \ d(x,y) = d(x,z) + d(z,y) \bigr\} .$$
Note that if $(\gamma_0,\ldots, \gamma_\ell)$ is a geodesic from $x$ to $y$ 
then $\{\gamma_0,\ldots,\gamma_\ell\} \subseteq [x,y]$.

The following is a simple but important observation.

\begin{lem}
\label{lem:interval}
Let $(X,d)$ be a metric space.
For any $y \in X$ and any $z \in [x_0,y]$ we have that $b_z \geq b_y$.
\end{lem}

\begin{proof}
Compute:
\begin{align*}
b_z(x) & = d(z,x) - d(z,x_0) = d(z,x) - d(y,x_0) + d(z,y) \geq d(y,x) - d(y,x_0) = b_y(x).\; \qedhere
\end{align*}
\end{proof}

\begin{lem} \label{lem:compactness of geodesics}
Let $\gamma^{(n)} = (1=\gamma_0, \ldots,\gamma_{\ell_n})$
be a sequence of geodesics starting at $1$ and of growing length $\ell_n \to \infty$.

Then, there exists a subsequence $(n_k)_k$ such that $(\gamma^{(n_k)} )_k$
converge to an infinite geodesic $\gamma$ starting at $\gamma_0=1$, in the following sense:
For all $\ell$, there is some  $k_0(\ell)$ such that for all $k \geq k_0(\ell)$ we have 
$(\gamma_0,\ldots, \gamma_\ell) = (\gamma^{(n_k)}_0 , \ldots, \gamma^{(n_k)}_\ell)$.
\end{lem}

\begin{proof}
This is just the compactness of the space $\{0,1\}^G$
with respect to pointwise convergence.

Let $F_n = \{ \gamma^{(n)}_0 , \ldots, \gamma^{(n)}_{\ell_n } \}$, and $f_n = \ind{F_n}$.
Compactness of $\{0,1\}^G$ provides a subsequence $(f_{n_k})_k$ converging pointwise to 
$f_{n_k} \to f \in \{0,1\}^G$.  
For any $x \in f^{-1}(1)$ there is some $k_0(x)$ such that 
for all $k \geq k_0(x)$ we have $f_{n_k}(x) = 1$.
Since $(\gamma^{(n)}_0 , \ldots, \gamma^{(n)}_{\ell_n })$
is a geodesic, it must be that $\gamma^{(n_k)}_{|x|} = x$ for all $k \geq k_0(x)$.
Setting $\gamma_{|x|} = x$ for all $x \in f^{-1}(1)$ provides us with the required infinite geodesic.
\end{proof}





As a consequence we have:

\begin{lem} \label{lem:Busemann dominates}
Let $d_S$ be a Cayley metric on a group $G$. For any $h \in \p (G,d_S)$, there exists a Busemann point $\gamma_\infty \in \p_b (G,d_S)$ such that $\gamma_\infty(x) \geq h(x)$ for all $x\in G$.
\end{lem}

\begin{proof}
Choose a sequence $(y_n)_n$ such that $b_{y_n}\to h$ in $\p(G,d_s)$. In particular, we have $\ell_n:= |y_n| \to \infty$. Let $\gamma^{(n)}$ be a geodesic from $1$ to $y_n$. Passing to a subsequence, we may suppose that $(\gamma^{(n)})$ converges to an infinite geodesic $\gamma$ in the sense of Lemma \ref{lem:compactness of geodesics}.

For each $x\in G$, there exists $r$ such that $b_{\gamma_r}(x) = \gamma_\infty(x)$.
Moreover, for all $n$ large enough, we have $\gamma_r=\gamma_r^{(n)}\in [1,y_n]$. By Lemma \ref{lem:interval}, we have
$$ \gamma_\infty(x) = b_{\gamma_r}(x) \geq b_{y_n}(x) \to h(x) $$
hence $\gamma_\infty(x)\ge h(x)$.
\end{proof}
We now prove Theorem \ref{thm:Busemann annihilator}, stating that $N_{\p_b (G,d_S)} = N$ for Cayley metrics $d_S$ on $G$.

\begin{proof}[Proof of  Theorem \ref{thm:Busemann annihilator}]
Let $M=N_{\p_b (G,d_S) }$.
It suffices to prove that $M \subseteq N$, since the other direction always holds.

Let $x \not\in N$.  
Since $M$ is a subgroup, it suffices to prove that either $x \not\in M$ or $x^{-1} \not\in M$.

Since $x\notin N$, there exists $h \in \p (G,d_S)$ such that $h(x) \neq 0$. Let us first suppose that $h(x)>0$. By Lemma \ref{lem:Busemann dominates}, there is a Busemann point $\gamma_\infty \in \p_b (G,d_S)$ such that $\gamma_\infty(x) \geq h(x) >0$, implying that $x \not\in M$. If instead $h(x) < 0$, then
$$x^{-1}. h(x^{-1})=-h(x)>0$$
and the same argument implies that $x^{-1}\notin M$.
\end{proof}

\section{The annihilator is locally finite} \label{scn:locally_finite}

In this section, we prove Proposition \ref{prop:locally finite}, stating that the annihilator $N$ is a locally finite group. We start with a lemma characterising when an element $x\in G$ belongs to $N$.
\begin{lem} \label{lem:indistinguishable}
    Let $d$ be a proper integer-valued metric on $G$. The following are equivalent:
    \begin{enumerate}[topsep=-2mm, leftmargin=8mm, itemsep=2pt, label=(\arabic*)]
        \item $x\in N$.
        \item $d(x,y)=d(1,y)$ for all but finitely many $y\in G$.
        \item There exists $r(x)\ge 0$ such that $d(x,y)=d(1,y)$ for all $\abs y\ge r(x)$.
    \end{enumerate}
\end{lem}
\begin{proof}
    We prove $\neg(1)\Leftrightarrow\neg(2)$ and $(2)\Leftrightarrow (3)$. 
    
    Let $x \notin N$. There exists $h\in\p(G,d)$ such that $h(x)\ne 0$, and $y_n\in G$ such that $b_{y_n}\to h$. We have $d(x,y_n)-d(1,y_n)=b_{y_n}(x)=h(x)\ne 0$ for all $n\ge n_0$, and $\{y_n:n\ge n_0\}$ must be infinite (otherwise $h=b_y$ for some $y\in G$, see \cite[Proposition 2.2]{RY23}).
    
    Suppose that we can find a sequence $(y_n)_n$ with $|y_n| \to \infty$ and $b_{y_n}(x) \ne 0$ for all $n$. Since $d$ is integer-valued, we have $\abs{b_{y_n}(x)}\ge 1$. We can extract a subsequence $(y_{n_k})_k$ such that $b_{y_{n_k}} \to h \in \p (G,d)$. We obtain that $\abs{h(x)}\ge 1$, proving that $x \notin N$.

    Finally $(2)\Leftrightarrow(3)$ follows from $(G,d)$ being discrete and proper.
\end{proof}
\begin{proof}[Proof of Proposition \ref{prop:locally finite}]

We consider a finite subset $U\subseteq N$ such that $U=U^{-1}$, and prove that $H:=\langle U\rangle$ is finite. Let $r = \max \{ r(u) \ : \ u \in U \}$.

Consider the set $A=\{ y \in H \ : \ \abs y>r \}$.
Assume for a contradiction that $A \neq \emptyset$.
Then, we can choose $y \in A$ so that $|y|_U$ is minimal.
That is, $|z|_U \geq |y|_U$ for all $z \in A$.

Now, we can choose $z \in H$ and $u \in U$ such that $|y|_U = |z|_U+ 1$ and $y=uz$.
Because $|y| > r \geq r(u)$, we have that 
$|z| = d(u,y) = d(1,y)= |y|$, so $z \in A$, a contradiction!

We conclude that $H \subseteq \{ x \in G \ : \ |x| \le r \}$, which is a finite set.
\end{proof}
\begin{remark}
The assumption that $d$ takes values in a uniformly discrete set cannot be dropped. For instance, for $G=H_3(\Z)$ and $d$ any homogeneous metric on $H_3(\R)$, we have $N_{\partial(G,d)}=[G,G]\simeq\Z$ by \cite[Proposition 2.4]{fisher2021sub}. In contrast, since $H_3(\Z)$ is torsion-free, we have $N_{\partial(G,d)}=\{1\}$ for any word metric or any integer-valued metric $d$.
\end{remark}
Finally, we give an example of integer-valued metric for which $N$ is infinite.
\begin{prop}
    Consider the lamplighter group $G=\Z/2\Z\wr\Z$.
    \begin{enumerate}[topsep=-2mm, leftmargin=8mm, itemsep=2pt, label=(\arabic*)]
        \item For all word metric $d_S$ on $G$, the annihilator $N$ is finite.
        \item There exists a proper integer-valued metric $d$ such that $N=\bigoplus_{\Z}\Z/2\Z$ is infinite.
    \end{enumerate}
    More generally, for any finitely generated group $G$ with a locally finite subgroup $H$, there exists a proper integer-valued metric $d$ on $G$ such that $N\ge H$.
\end{prop}
\begin{proof}
Elements of $\Z/2\Z\wr\Z$ are pairs $(\Phi,k)$ where $\Phi\colon\Z\to\Z/2\Z$ is a finitely supported function, and $m\in\Z$ is an integer. Fix a finite generating set $S=S^{-1}$. We define
\begin{align*}
    a_S & = \min\left(\bigcup_{(\Phi,k)\in S}\supp\Phi\right), \quad b_S = \max\left(\bigcup_{(\Phi,k)\in S}\supp\Phi\right), \\[1mm]
    \ell & = \max\bigl\{\abs k:(\Phi,k)\in S\bigr\}.
\end{align*}
Consider an element $g=(\Psi,\ell)\in S$ and let $y=g^n$. Then $d_S\bigl((\Phi,k),y\bigr)=d_S(1,y)$ forces
\begin{itemize}[topsep=-2mm, parsep=2mm, leftmargin=5mm]
    \item $\supp(\Phi)\subset[a_S,+\infty)$ and $k\ge 0$ when $n\ge 0$, and
    \item $\supp(\Phi)\subset(-\infty,b_S]$ and $k\le 0$ when $n\le 0$.
\end{itemize}
Therefore, if $x=(\Phi,k)\in N$, then $\supp(\Phi)\subseteq[a_S,b_S]$ and $k=0$, which is a finite set.
\medbreak
We fix $G$ a finitely generated group and $H$ a locally finite subgroup. We can exhaust $H$ as an increasing union of finite subgroups $F_n\nearrow H$. For instance, we can take $G=\Z/2\Z\wr \Z$, $H=\bigoplus_{m\in\Z}\Z/2\Z$ and $F_n=\bigoplus_{\abs{m}\le n} \Z/2\Z$. We define increasingly large finite subsets that will serve as balls for the metric $d$:
\begin{itemize}[topsep=-2mm, parsep=2mm, leftmargin=5mm]
    \item $B_0=\{1_G\}$,
    \item $B_1\ni 1_G$ is a finite symmetric set such that $G=\langle B_1\cup H\rangle$, and
    \item $B_n= F_n\Bigl(\bigcup_{k=1}^{n-1} B_kB_{n-k}\Bigl)F_n$ for all $n\ge 2$.
\end{itemize}
Finally, we define $d(g,h)=n$ if $g^{-1}h\in B_n\setminus B_{n-1}$. The triangle inequality is satisfied since $B_mB_n\subseteq B_{m+n}$ by construction. For instance, for $G=\Z/2\Z\wr\Z$, if we take $B_1=\{(\mathbf 0,0),(\mathbf 0,\pm 1)\}$ and the previous exhaustion $F_n$, then
\[ \forall n\ge2, \quad B_n = \bigl\{ (\Phi,k) : \abs k\le n,\; \supp\Phi\subseteq B(0,n)\cup B(k,n)\bigr\}\]
where $B(c,n)$ is the ball of centre $c$ and radius $n$ in $\Z$.

We finally prove that $H\le N$. Fix $f\in H$, say $f\in F_n$. Then, for all $g\in G$ such that $\abs{g}\ge n$, we have $f\in F_n\subseteq F_{\abs g}$ and $g\in B_{\abs g}$ hence
\[ f^{-1}g\in F_{\abs g} B_{\abs g}=B_{\abs g}, \]
i.e., $\abs{f^{-1}g}\le \abs{g}$. The same argument gives $\abs{g}\le \abs{f^{-1}g}$ as soon as $\abs{f^{-1}g}\ge n$. Therefore, if $\abs{g}\ge n$ and $\abs{f^{-1}g}\ge n$ (which is true for all but finitely many elements $g\in G$), then $d(g,f)=d(g,1)$. It follows that $f\in N$, hence $N\ge H$.
\end{proof}
\section{Virtually abelian groups}

\label{scn:VA groups}

In this section we prove Theorem \ref{thm:virt Abelian}.
Since Cayley graphs of virtually-$\Z$ groups always have finite metric-functional boundaries (see \cite{RY23}), it suffices to prove the following. 

\begin{thm}
\label{thm:infinite Busemann boundary}
Let $G$ be a group and assume that $H$ is a finite index subgroup $[G:H] < \infty$
such that $H \cong \Z^d$ for $d \geq 2$.

Then, $|\p_b (G, d_S)| = \infty$ for any Cayley metric $d_S$ on $G$.
\end{thm}

In order to prove this theorem, we will prove the following lemma (compare to Lemmas 3.4 and 3.5 in \cite{RY23}):

\begin{lem} \label{lem:good homomorphism}
Let $G$ be a group and assume that $H$ is a finite index subgroup $[G:H] < \infty$
such that $H \cong \Z^d$.

Assume that $f\colon H \to \mathbb{R}$ is a $1$-Lipschitz homomorphism (so that $f(x) \leq |x|_S$ for every $x \in H$) 
with the following properties:
There exists $x \in H$ and a positive integer $p$ such that $f(x^p)= |x^p |_S$. 
Also, for any $y \in H$ we have that if $f(y)=|y|_S$ then $y \in \langle x \rangle$.

Then the following hold:
\begin{itemize}[leftmargin=5mm, itemsep=5pt]
    \item For every $y \in H$, the sequence $(b_{y x^{pn}})_n$ converges to a Busemann point.
    \item Two sequences $(b_{y x^{pn}})_n$ and $(b_{z x^{pn}})_n$ converge to the same Busemann point only if $y^{-1} z \in \langle x \rangle$.
\end{itemize}
In particular, if $d \geq 2$ then $|\partial_b (G,d_S)|=\infty$.
\end{lem}

\begin{proof}
    Write $w=x^p$.
    By Lemma 3.4 in \cite{RY23}, the assumption that $f(w) = |w|_S$ implies that there exists a geodesic $\gamma$ in $\Gamma(G,S)$ such that $\gamma_{n|w|_S}=w^n$ for every $n$, as required.  So for every $y \in H$, $b_{yw^n}=y.b_{w^n}$ converges to $y.\gamma_{\infty} \in \partial_b(G,d_S)$.

    Assume that $b_{yw^n}$ and $b_{zw^n}$ converge to the same Busemann point, and let $\gamma$ be as above. Proposition 2.1 in \cite{Walsh} tells us that since the two geodesics $(yw^n)_n$ and $(z w^n)_n$ converge to the 
same limit, there exists a third geodesic that intersects each of them infinitely many times.
    Specifically, there exist $n,m,k$ such that $n \geq m$ and
    $$d_S(yw^m,zw^k)+d_S(zw^k,yw^n)=d_S(yw^m,yw^n) = d_S(w^m,w^n) . $$
    Also, since $f$ is a $1$-Lipschtiz homomorphism,
    $$f(w^{-m}y^{-1}zw^k) \leq d_S(yw^m,zw^k) \AND f(w^{-k}z^{-1} y w^n) \leq d_S(zw^k,yw^n)$$
    so that
    $$d_S(yw^m,zw^k)+d_S(zw^k,yw^n) \geq f(w^{-m}y^{-1} z w^k)+f(w^{-k} z^{-1} yw^n)=f(w^{n-m}) .$$
    Since $f(w^{n-m})=(n-m)|w|_S = d_S(w^m, w^n)$ we get equality throughout:
    $$f(w^{-m}y^{-1}zw^k) = |w^{-m}y^{-1}zw^k|_S
    \AND f(w^{-k}z^{-1} y w^n) = |w^{-k}z^{-1} y w^n|_S . $$
By our assumptions on the function $f$, 
this implies that 
$w^{-m} y^{-1} z w^k \in \langle x \rangle$ and thus $y^{-1} z \in \langle x \rangle$, as required.

Finally, since $\langle x \rangle$ is cyclic and $H \cong \mathbb{Z}^d$,  if $d \geq 2$ we have $[H:\langle x \rangle]=\infty$ and so the set $\{y.\gamma_{\infty} \ | \  y \in H\} \subset \partial_b(G,d_S)$ is infinite. 
\end{proof}

\begin{proof}[Proof of Theorem \ref{thm:infinite Busemann boundary}]
By the Lemma \ref{lem:good homomorphism}, 
it is enough to prove the existence of a $1$-Lipschitz homomorphism $f\colon H \to \mathbb{R}$ and some $x \in H$
and positive integer $p$ 
such that $f(x^p) = |x^p|_S$ and 
$f(y) = |y|_S$ implies $y \in \langle x \rangle$.

Without loss of generality, we may assume that $H$ is a normal subgroup of $G$. 
Let $\xi\colon H \to \mathbb{Z}^d$ be an isomorphism. Since $H$ is normal, for every $g \in G$ the map $x \mapsto gxg^{-1}$ is an automorphism of $H$. In fact, this defines a representation $\pi\colon G \to \mathrm{GL}_d(\mathbb{Z}) \subset \mathrm{GL}_d(\mathbb{R})$ such that for all $g \in G$ and $x \in H$:
$$\pi_g \xi(x)=\xi(gxg^{-1})$$

Consider the Cayley graph $\Gamma(H \backslash G,S)$ and let $C$ be the set of all 
$x \in H$ such that $x=s_1s_2 \cdots s_n$ and $(s_1,s_2,\ldots,s_n)$ labels a simple cycle in $\Gamma(H \backslash G,S)$. (Precisely, $s_1 \cdots s_n \in C$ if $s_1 \cdots s_n \in H$ and $H s_1 \cdots s_j \neq H s_1 \cdots s_i$ for all $i \neq j$.
If some $s \in S$ is in $H$ then it labels a simple cycle of length $1$, so $s \in C$.) 
Note that $C$ is non-empty and finite because $\Gamma(H \backslash G,S)$ is a finite connected graph.

Recall the following fact: let $F \subset \R^d$ be a finite set and let $P=\conv(F)$ be its convex hull. An extreme point of $P$ is a point $e\in P$ such that there exists a linear functional $\phi_e\colon \R^d \to \R$ such that $\phi_e(p) \leq 1$ for all $p \in P$ and $\phi_e(p)=1$ if and only if $p=e$.

Define:
$$F=\left\{\frac{\xi(gxg^{-1})}{|x|_S} \;\middle|\; g \in G ~,~ x \in C\right\}=\bigcup_{g \in G} \left\{\pi_g\left(\frac{\xi(x)}{|x|_S}\right) \;\middle|\; x \in C\right\}$$
The set $F$ is finite because $C$ is finite, $[G:H]<\infty$ and $H$ is abelian (so $gxg^{-1}$ only depends on $x$ and the coset of $g$ in $H \backslash G$). 
Note that $F$ is invariant under $\pi_g$ for all $g \in G$.

Define the polytope $P=\conv(F)$. By construction, $P$ is $\pi$-invariant, and in particular for every extreme point $e \in P$ and every $g \in G$, $\pi_g(e)$ is also an extreme of $P$.

If $x \in C$ then $\frac{\xi(x)}{|x|_S} \in P$ by definition. 

Also, if $(s_1,\ldots, s_n)$ labels a simple cycle in $\Gamma ( H \backslash G, S)$, 
then so does $(s_n^{-1}, \ldots, s_1^{-1})$. 
Therefore, $x \in C$ implies that $x^{-1} \in C$.

Note that $\frac{\xi(x^{-1})}{|x^{-1}|_S}=-\frac{\xi(x)}{|x|_S}$ and so by convexity of $P$ we get that $0 \in P$.

{\bf Step I.}
We now show, by induction on $|x|_S$, that $\frac{\xi(x)}{|x|_S} \in P$ for every $x \in H$. 

The base case is for $x \in H$ with minimal $|x|_S$; \ie $|x|_S \leq |y|_S$ for all $y \in H$.
In this case, we can write $x = s_1 \cdots s_n$ for $n=|x|_S$.
If $(s_1, \ldots, s_n)$ does not label a simple cycle, then there are some $i < j$ such that 
$H s_1 \cdots s_j = H s_1 \cdots s_i$ implying that $(s_{i+1}, \ldots, s_j)$ labels a cycle,
and specifically $s_{i+1} \cdots s_j \in H$.  But $|s_{i+1} \cdots s_j|_S = j-i < n=|x|_S$,
contradicting the minimality of $|x|_S$.  Therefore, if $|x|_S$ is minimal, then $x \in C$,
so $\frac{ \xi(x) }{ |x|_S} \in P$ by definition.

Assume that $n$ is such that there exists $y \in H$ with $|y|_S < n$ and for all $y \in H$
with $|y|_S < n$ it holds that $\frac{ \xi(y) }{ |y|_S} \in P$.
Let $x=s_1s_2 \cdots s_n \in H$ be such that $|x|_S=n$. 
Since $x \in H$, the sequence $(s_1,s_2,...,s_n)$ labels a cycle in $\Gamma(H \backslash G,S)$. 
If it is a simple cycle, $\frac{\xi(x)}{|x|_S} \in P$ and we are done. 
If not, as before, there exist $1 \leq i<j \leq n$ such that $(s_{i+1},\ldots,s_j)$ labels a simple cycle. 
Denote $g=s_1 \cdots s_i$, $y=s_{i+1} \cdots s_j$ and $z=gs_{j+1} \cdots s_n$, so that $x=gyg^{-1}z$ and $y,z \in H$. 
Since $|y|_S \leq j-i$ and $|z|_S \leq i+n-j$, we get $|y|_S+|z|_S \leq n=|x|_S$.

Note that $y \in C$ and $|z|_S < n$. 
By the induction hypothesis we get that both $\frac{\xi(gyg^{-1})}{|y|_S} = \pi_g \big( \frac{ \xi(y)}{|y|_S} \big)$ 
and $\frac{\xi(z)}{|z|_S}$ are in $P$. 
Since $|y|_S+|z|_S \leq |x|_S$ and $0 \in P$ we get that 
$$\frac{\xi(x)}{|x|_S}=\frac{|y|_S}{|x|_S} \cdot \frac{\xi(gyg^{-1})}{|y|_S}+\frac{|z|_S}{|x|_S} \cdot \frac{\xi(z)}{|z|_S} \in P $$
by convexity.

{\bf Step II.}
Now, let $e' \in P$ be an extreme point.  
Note that $e'$ must be in $F$, \ie $e'=\frac{\xi(gwg^{-1})}{|w|_S}$ for some $g \in G$ and $w \in C$. 
Define $e = \pi_{g^{-1}}(e') = \frac{ \xi(w) }{ |w|_S }$, and note that $e$ is also an extreme point of $P$. 

Let $\phi_e\colon \mathbb{R}^d \to \mathbb{R}$ be a linear functional such that $\phi_e(p) \leq 1$ for every $p \in P$, and $\phi_e(p)=1$ if and only if $p=e$. 
Define $f\colon H \to \mathbb{R}$ by $f(x)=\phi_e(\xi(x))$. Since $\xi$ is an isomorphism and $\phi_e$ is linear, $f$ is a homomorphism. Also, 
for any $y,z \in H$, because $\frac{ \xi(z^{-1} y) }{ |z^{-1} y|_S } \in P$,
$$ f(y)-f(z) =  f(z^{-1} y) = \phi_e(\xi(z^{-1} y)) \leq |z^{-1} y|_S=d_S(y,z) , $$
and reversing the roles of $y,z$ we also have 
$f(z)-f(y) \leq d_S(y,z)$, 
so $f$ is $1$-Lipschitz.

It remains to show that there exists $x \in H$ and a positive integer $p$ 
such that $f(x^p) = |x^p|_S$ and 
if $f(y)=|y|_S$ then $y \in \langle x \rangle$.

Let $M$ be the maximal cyclic subgroup containing $\xi(w)$,
and let $x \in H$ be such that $\xi(x)$ generates $M$.
By possibly passing to $x^{-1}$ instead of $x$ we can assume without loss of generality that 
there is a positive integer $p$ such that $w=x^p$.
Specifically,  $f(x^p) = |x^p|_S$.

Now, let $y \in H$ be such that $f(y)=|y|_S$. 
So $\phi_e\big(\frac{\xi(y)}{|y|_S}\big) = 1$ implying that $\frac{\xi(y)}{|y|_S}=e=\frac{\xi(w)}{|w|_S}$. 
Let $\frac{|y|_S}{|w|_S}=\frac{a}{b}$ with $a,b$ coprime positive integers. 
Since $b\xi(y)=a\xi(w)$, since $a,b$ are coprime, and since $\xi(w) \in \Z^d$, 
we get that $\frac{1}{b}\xi(w) \in \mathbb{Z}^d$, so that $\frac1b \xi(w) \in M$ (by the maximality of $M$).
Therefore, $\xi(y) = a \cdot \frac{1}{b}\xi(w) \in M$, so that $y \in \langle x \rangle$ (recall that $\xi$ is an isomorphism).
\end{proof}

As a consequence we can prove the following, which may be of independent interest.

\begin{lem}
\label{lem:reduced finite boundary}
Let $d_S$ be a Cayley metric on a group $G$ such that $|\p_b(G,d_S)| < \infty$.
Let $K = \{ x \in G : \forall h \in \p_b(G,d_S), \ x.h=h \}$ be the kernel of the action of $G$ 
on $\p_b(G,d_S)$.

Then, for any $g,h \in \p_b(G,d_S)$ there exists $q \in \{\pm 1\}$ such that 
$g \big|_K \equiv q \cdot h \big|_K$.

As a consequence, $|g(x)-q\cdot h(x)| \leq 2 [G:K]$ for all $x \in G$.
\end{lem}

\begin{proof}

When $|\partial_b(G,d_S)|<\infty$ we have that $[G:K]<\infty$.
Let $g,h \in \p_b (G,d_S)$ and define $\varphi\colon K \to \mathbb{Z}^2$ by $\varphi(x)=(g(x),h(x))$. This is a homomorphism since $g(x^{-1}y)=g(x^{-1})+x.g(y)=g(x^{-1})+g(y)$ for $x,y\in K$. It is simple to verify that both $g$ and $h$ are unbounded \cite[Proposition 2.2]{RY23}, hence $\varphi$ is a non-trivial homomorphism.

Since $[G:K] < \infty$, we know that $K$ is finitely generated. 
Let $V = \mathsf{span}\ \varphi(K)$ be the vector space over $\mathbb{Q}$ spanned by the vectors $\{ \varphi(k) : k\in K \}$.
We claim that $\dim V = 1$. 

To prove this claim, assume for a contradiction that $\dim V \geq 2$. Since $V \leq \mathbb{Q}^2$ it must be that $V =\mathbb{Q}^2$. This implies that $\varphi(K)$
is a finitely generated torsion-free Abelian group 
of rank greater than $1$.
So $K$ has a quotient isomorphic to $\mathbb{Z}^2$. Note that $[K,K]$ is normal in $G$ since $K$ is normal in $G$ and $[K,K]$ is characteristic in $K$. It follows that
$$G/[K,K]\ge_{\mathrm{f.i.}} K/[K,K] \simeq \Z^d\times T$$
with $d\ge 2$ and $T$ finite, that is, $G$ has a quotient which is virtually-$\Z^d$ for $d \geq 2$. Combining Theorem \ref{thm:infinite Busemann boundary} and Corollary 1.11 in \cite{BT24}, we get that $|\partial_b(G,d_S)|=\infty$, a contradiction!
Since $\vphi$ is a non-trivial homomorphism, we conclude that $\dim V = 1$.

Therefore, there exists $q \in \mathbb{Q}\setminus\{0\}$ such that $g\big|_K \equiv q \cdot h\big|_K$.

Let $x \in G$.  Since the diameter of the quotient Cayley graph $\Gamma(K\backslash G,S)$ is bounded by its size $[G:H]$, there exists $y\in K$ such that $d_S(x,y)\le [G:H]$. It follows that
\begin{align*}
|g(x)-q\cdot h(x)|
& = \abs{g(x)-g(y)+q\cdot h(y)-q\cdot h(x)} \\
& \leq |g(x)-g(y)|+|q| \cdot |h(y)-h(x)| \\
& \leq (1+|q|) \cdot d_S(x,y) \\
& \leq (1+|q|) \cdot [G:K]
\end{align*}
using that $g$ and $h$ are $1$-Lipschitz, and this bound holds uniformly over $x \in G$.

We are left with showing that $q \in \{\pm 1\}$. 

Let $\gamma$ and $\eta$ be geodesics such that $g=\gamma_\infty$ and $h=\eta_\infty$. For every $m,n \in \N$ we have
$$\Big|d_S(\gamma_m,\eta_n)-|\gamma_m|_S\Big| \leq |\eta_n|_S=n$$
On the other side, since $b_{\gamma_m} \to g$, there exists $m_0$ such that for all $m \geq m_0$,
$$ \Big| d_S(\gamma_m,\eta_n)-|\gamma_m|_S \Big| = | g(\eta_n)| = \abs{\big(g(\eta_n)-qh(\eta_n)\bigl) + qh(\eta_n)}= \abs{\big(g(\eta_n)-qh(\eta_n)\bigl) - qn}. $$
Putting everything together, we have
\[ \abs{\frac{g(\eta_n)-q\cdot h(\eta_n)}n-q\,} \le 1 \]
with the numerator $g(\eta_n)-qh(\eta_n)$ bounded by a constant $(1+\abs q)\cdot [G:K]$, hence the first term tends to $0$ as $n\to\infty$ and we conclude that $\abs q\le 1$.

We now repeat the argument with the roles of $g,h$ interchanged, and $\frac1q$ instead of $q$.
(Indeed, $ h\big|_K \equiv q^{-1} \cdot g\big|_K$ and therefore $|h(x) - q^{-1} g(x)| \leq (1+\tfrac{1}{|q|}) \cdot [G:K ]$
for all $x \in G$.) This yields $\frac{1}{|q|} \leq 1$ and finally $|q| = 1$, completing the proof.
\end{proof}

\section{Proof of main results}

\label{scn:main results}

In this section we prove Lemma \ref{lem:N finite} and Theorem \ref{thm:main}.

Before the formal proof of Lemma \ref{lem:N finite}, 
we provide the idea behind our proof, 
which comes from the following apparent contradiction: given two infinite geodesics $\beta,\gamma$, the previous Lemma \ref{lem:reduced finite boundary} tells us that
\begin{equation} \label{eq:unit_speed}
\forall m\ge0,\quad \gamma_\infty(\beta_m) = qm + O(1)
\end{equation}
for some $q\in\{\pm1\}$. (Take $g=\gamma_\infty$ and $h=\beta_\infty$, and evaluate at $x=\beta_m$.) Meanwhile, if the annihilator $N$ is large or even infinite, we will be able to find long finite geodesics $\alpha^{(n)}$ such that $\gamma_\infty(\alpha^{(n)}_\ell)=0$ and not $\pm\ell+O(1)$ as Equation (\ref{eq:unit_speed}) would predict. A naive idea is to consider an infinite geodesic $\beta$ which is a limit point of the finite geodesics $\alpha^{(n)}$ in the hope to derive a contradiction. However, such a limit point may well satisfy Equation (\ref{eq:unit_speed}), if the first failure to (\ref{eq:unit_speed}) happens later and later, see Figure \ref{fig:geodesics_limiting}.

\begin{center}
    \begin{tikzpicture}
        \definecolor{color0}{RGB}{216,27,96} 
        \definecolor{color1}{RGB}{255,193,7}
        \definecolor{color2}{RGB}{30,136,229}
            
        \draw[thick, rounded corners=10] (-4,1.15) rectangle (4,3.35);
        \node at (4.2,1.1) {$G$};
        \draw[color0] (0,2.25) ellipse (.3cm and 1.1cm);
        \node[color0] at (0,3.6) {$N$};
        
        \node[circle, fill=black, inner sep=1pt, label=above:$1$] at (0,1.5) {};

        \draw[color2!40, -latex] (0,1.5)    to[out=12, in=-5, looseness=10] (0,1.9)
                                        to [out=175, in=-175, looseness=15] (-.05,2.2);
        \draw[color2!60, -latex] (0,1.5)    to[out=9, in=-70, looseness=1] (2,2.25)
                                        to [out=110, in=-5, looseness=1] (.05,2.4);
        \draw[color2!80, -latex] (0,1.5)    to[out=3, in=-40, looseness=1] (3,2.2)
                                        to [out=150, in=-150, looseness=1] (3,2.4)
                                        to [out=30, in=0, looseness=1] (0,2.8);
        \draw[color2, -latex] (0,1.5) to[out=-2, in=-178] (3.8,1.4);
        \node at (-.7,2.4) {\small$\alpha^{(1)}$};
        \node at (.65,3.1) {\small$\alpha^{(3)}$};
        \node at (3.6,1.6) {\small$\beta$};
        
        \draw[thick, dotted, -latex] (0,1) -- (0,.15);
        \node at (.35,.6) {$\gamma_\infty$};
        
        \draw[very thick, -latex] (-4,0) -- (4,0);
        \node[circle, fill=black, inner sep=1pt, label=below:$0$] at (0,0) {};
        \node at (4.2,-.3) {$\Z$};
    \end{tikzpicture}
    \captionsetup{font=small, margin=30mm}
    
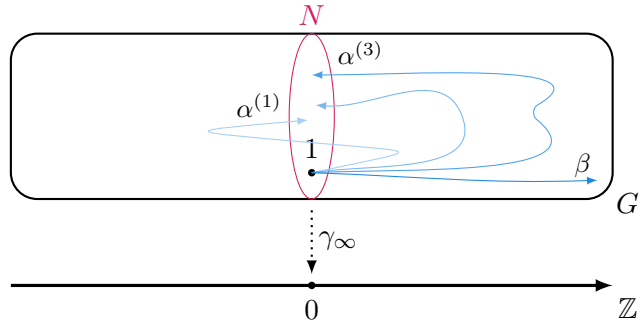
\captionof{figure}{Finite geodesics $\alpha^{(n)}$ with endpoints in $N$ limiting to an infinite geodesic $\beta$ satisfying Equation (\ref{eq:unit_speed}).}
    \label{fig:geodesics_limiting}
\end{center}
The next idea is to start from the first ``bend'' in each geodesics $\alpha^{(n)}$ where Equation (\ref{eq:unit_speed}) is not locally satisfied, translate the remaining part back to $1$ to get long finite geodesics that ``start slowly with respect to $\gamma_\infty$'' (see Figure \ref{fig:bad_start}), and then take a limit point. Producing these long finite geodesics that ``start slowly'' is the content of Lemma \ref*{lem:bounded values}.
\begin{center}
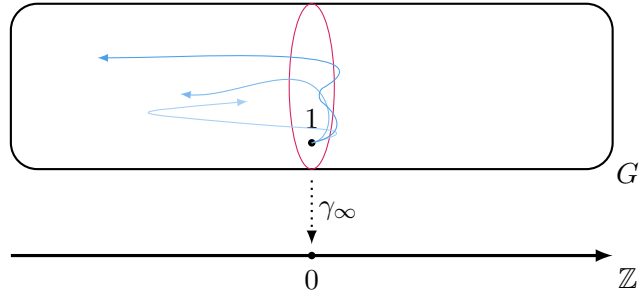

    \begin{tikzpicture}
        \definecolor{color0}{RGB}{216,27,96} 
        \definecolor{color1}{RGB}{255,193,7}
        \definecolor{color2}{RGB}{30,136,229}
        
        \draw[thick, rounded corners=10] (-4,1.15) rectangle (4,3.35);
        \node at (4.2,1.1) {$G$};
        \draw[color0] (0,2.25) ellipse (.3cm and 1.1cm);
        
        \node[circle, fill=black, inner sep=1pt, label=above:$1$] at (0,1.5) {};
                                        
        \draw[color2!40, shift={(-.8,-.15)}, -latex] (.8,1.65)    to[out=14, in=-5, looseness=2.5] (0,1.9)
                                        to [out=175, in=-175, looseness=15] (-.05,2.2);
        \draw[color2!60, shift={(-1.8,-.25)}, -latex] (1.8,1.75)    to[out=9, in=-70, looseness=1] (2,2.25)
                                        to [out=110, in=-5, looseness=1] (.05,2.4);
        \draw[color2!80, shift={(-2.85,-.17)}, -latex] (2.85,1.67)    to[out=12, in=-40, looseness=1.8] (3,2.2)
                                        to [out=150, in=-150, looseness=1] (3,2.4)
                                        to [out=30, in=0, looseness=1] (0,2.8);
        
        \draw[thick, dotted, -latex] (0,1) -- (0,.15);
        \node at (.35,.6) {$\gamma_\infty$};
        
        \draw[very thick, -latex] (-4,0) -- (4,0);
        \node[circle, fill=black, inner sep=1pt, label=below:$0$] at (0,0) {};
        \node at (4.2,-.3) {$\Z$};
    \end{tikzpicture}
    \captionsetup{font=small, margin=30mm}
    \captionof{figure}{The geodesics $\alpha^{(n)}$, with their ``bend'' translated back to $1$ so that Equation (\ref{eq:unit_speed}) fails from the start.}
    \label{fig:bad_start}
\end{center}
\begin{lem} \label{lem:bounded values}
Let $d_S$ be a Cayley metric on a group $G$ such that the set of Busemann points $\p_b (G,d_S)$ is finite. Recall that $K = \{ x \in G \ : \ \forall \ h \in \p_b(G,d_S) \ , \ x.h=h \}$ is the kernel of the action of $G$ on $\p_b(G,d_S)$, and $N = N_{\p_b (G,d_S)}$ is the annihilator.

Then, for any $m> 0$ and any $\ell$ satisfying
\[ m \le \ell \le \frac2{m+2}\sup \bigl\{ |x|_S : x \in N \bigr\}, \] 
there exists a finite geodesic $(1=\beta_0, \ldots, \beta_{\ell})$
such that $|\gamma_\infty(\beta_m) | \leq 6 [ G:K ] +1$
for any Busemann point $\gamma_\infty \in \p_b(G,d_S)$ 
\end{lem}

\begin{proof}
Fix some integers $0\le m\le \ell$ and $x \in N$ such that $\abs x_S \ge \frac{m\ell}2+\ell$. Under these hypothesis, there exists a positive integer $r=\left\lfloor\frac{\abs x}{m+2}\right\rfloor+1$ satisfying
\[ \abs x-m(r-1)\ge \ell\ge m \quad\text{and}\quad \abs x-mr<2r.\]
Let $(1 = \alpha_0, \ldots, \alpha_{|x|} = x)$ be a finite geodesic from $1$ to $x$.

Fix a Busemann point $g \in \p_b (G,d_S)$.
Consider $\vphi(k) = g(\alpha_{k+m}) - g(\alpha_k)$ for $k=0,\ldots, |x|-m$.
It is immediate that $|\vphi(k+1)-\vphi(k)| \leq 2$ for all $k$.

Let $\eps = \min \{ |\vphi(k)| : 0 \leq k \leq m(r-1) \}$, and $\sigma(k)$ be the sign of $\vphi(k)$; \ie $\sigma(k) = \1{ \vphi(k) \geq 0 } - \1{ \vphi(k) \leq 0 }$.
Note that if $\eps \geq 2$, 
then because $\vphi$ is $2$-Lipschitz, it must be that $\sigma(k) = \sigma(0)$ for all $0 \leq k \leq m(r-1)$.
Thus, under the assumption that $\eps \geq 2$,
we have
\begin{align*}
\sigma(0) \cdot g(\alpha_{mr}) & = \sum_{j=0}^{r-1} \sigma(mj) \cdot \vphi(mj) = \sum_{j=0}^{r-1} |\vphi(mj) | 
\geq 2 r .
\end{align*}
However, since $g$ is $1$-Lipschitz, and since $x \in N$,
\begin{align*}
0 & = \abs{g\bigl(\alpha_{|x|}\bigr)}
\ge \bigl| g(\alpha_{mr}) \bigr|
- \abs{ g\bigl(\alpha_{|x|}\bigr) - g(\alpha_{mr}) } \\
& \ge 2r - d_S\bigl(\alpha_{|x|} , \alpha_{m r}\bigr) = 2 r - \bigl( \abs{x} - mr\bigr)>0 ,
\end{align*}
a contradiction! We deduce that $\eps \leq 1$, \ie there exists $0 \leq t \le m(r-1)$ such that
$$|g(\alpha_{t+m}) - g(\alpha_t) | \leq 1.$$

For this $t$, let $h = \alpha_t.g \in \p_b(G,d_S)$.
By Lemma \ref{lem:reduced finite boundary}, 
there exists $q = q(g,h) \in \{-1,1\}$ such that $|g(y) - q \cdot h(y) | \leq 2 [G:K]$ for all $y \in G$.
In particular, for all $0 \leq j,k \leq |x|-t$,
\begin{align*}
\abs{g(\alpha_t^{-1} \alpha_{t+j}) - g(\alpha_t^{-1} \alpha_{t+k}) }
& = \bigl| h(\alpha_{t+j}) - h(\alpha_{t+k}) \bigr| \\
& \le \bigl| g(\alpha_{t+j}) - g(\alpha_{t+k}) \bigr| + 4[G:K] .
\end{align*}
In particular, for $(j,k)=(0,m)$, we have that $\abs{g(\alpha_t^{-1}\alpha_{t+m})-0}\le 4[G:K]+1$.
Note that $|x| - t \geq |x| - m(r-1) \geq \ell$ so we may define a geodesic by setting
$\beta_j = \alpha_t^{-1} \alpha_{t+j}$ for all $0 \leq j \leq \ell$.

Let $\gamma_\infty \in \p_b(G,d_S)$ be any Busemann point.  Again using Lemma  \ref{lem:reduced finite boundary}, 
there exists $p=q(\gamma_\infty,g) \in \{-1,1\}$ such that $|\gamma_\infty(y)-p\cdot g(y)| \leq 2[G:K]$.
This implies that 
$$ |\gamma_\infty(\beta_m)| \leq 2 [G:K] + | g(\alpha_t^{-1} \alpha_{t+m}) | \le 6[G:K] + 1 , $$
as required.
\end{proof}

\begin{proof}[Proof of Lemma \ref{lem:N finite}]
Let $K = \{ x \in G \ : \ \forall \ h \in \p_b(G,d_S) \ , \ x.h=h \}$ be the kernel of the action of $G$ on $\p_b(G,d_S)$.
Set $M = 6[G:K] + 1$.

Fix $m>M$ (\eg $m=M+1$) and assume for a contradiction that $N$ is infinite. 
Then, by Lemma \ref{lem:bounded values},
we can choose a finite geodesic $\beta^{(\ell)} = (1=\beta_0^{(\ell)}, \ldots, \beta_\ell^{(\ell)} )$  for each $\ell\ge m$ 
such that $|\gamma_\infty(\beta_m^{(\ell)}) | \leq M$ for any Busemann point $\gamma_\infty \in \p_b(G,d_S)$.


We now take a limit point $\gamma_\infty \in \p_b (G,d_S)$ of these geodesics $(\beta^{(\ell)} )_{\ell}$, as in Lemma \ref{lem:compactness of geodesics}.
This provides us with a subsequence $(\ell_k)_k$ such that for all $k$ 
we have $\gamma_m = \beta^{(\ell_k)}_m$. 

However, since $\gamma_\infty \in \p_b(G,d_S)$, we arrive at
$$ M < m = |\gamma_\infty(\gamma_m)| = |\gamma_\infty(\beta^{(\ell_k)}_m) | \leq M , $$
a contradiction!
\end{proof}

\begin{proof}[Proof of Theorem \ref{thm:main}]
As mentioned the fact that for any virtually-$\Z$ group all Cayley graphs have only finitely many Busemann points was 
shown in \cite{TY16}. 
So we only need to show that if for some Cayley metric $|\p_b(G,d_S)| < \infty$, then $G$ is virtually-$\Z$.

So assume that $G$ is a finitely generated infinite group with a Cayley metric $d_S$ such that 
the set of Busemann points is finite $|\p_b(G,d_S) | < \infty$.
By Lemma \ref{lem:N finite}, the annihilator $N= N_{\p(G,d_S)} = N_{\p_b(G,d_S)}$ is a finite subgroup of $G$.

Write $\p_b(G,d_S) = \{ h_1, \ldots, h_n \}$ and define
$\vphi\colon G \to \Z^n$ by $\vphi(x) = (h_1(x) , \ldots, h_n(x))$.
Let $K = \{ x \in G \ : \ \forall \ h \in \p_b (G,d_S) \ , \ x.h=h \}$ be the kernel of the action of $G$ on 
the Busemann points $\p_b (G,d_S)$.
It is easy to verify that $\vphi \big|_K$ is a homomorphism with kernel $\ker \vphi \big|_K = N \cap K$.
Since $\p_b(G,d_S)$ is finite we know that $[G:K] < \infty$,
and specifically, $K$ is an infinite group.

We conclude that $K / (K \cap N)$ is isomorphic to an infinite subgroup of $\Z^n$, and specifically 
$K / (K \cap N)$ is an infinite finitely generated Abelian group.  
It is well known that the fact that $K\cap N$ is finite then implies that $K$ is virtually Abelian.

(Indeed, since $K/(K \cap N)$ is Abelian, we have that $[K,K] \lhd K \cap N$ is a finite subgroup.
Therefore, for any $x \in K$ the set $\{ y^{-1} x y  \ : \ y \in K \}$ is finite (because $y^{-1} x y = [y,x^{-1} ] x \in [K,K]x$).
This implies that $C(x) : = \{ y \in K \ : \ [x,y]=1 \}$ is of finite index in $K$.
Since $K$ is finitely generated, taking generators $x_1, \ldots, x_r$ for $K$ we have
that the center of $K$ satisfies $Z(K) = \cap_{j=1}^r C(x_j)$, which is finite index in $K$ 
as a finite intersection of finite index subgroups.)

Since $[G:K]<\infty$ and $[K:Z(K)]<\infty$ we have that $Z(K)$ is an Abelian group of finite index in $G$.
Using Theorem \ref{thm:virt Abelian} with the virtually Abelian group $G$, we have that $|\p_b(G,d_S)| < \infty$
implies that $G$ is virtually $\Z$.
\end{proof}

\bibliographystyle{abbrv}
\bibliography{finite}

@article{TY16,
  title={Horofunctions on graphs of linear growth},
  author={Tointon, Matthew and Yadin, Ariel},
  journal={Comptes Rendus Mathematique},
  volume={354},
  number={12},
  pages={1151--1154},
  year={2016},
  publisher={Elsevier}
}

@article{develin,
  title={{C}ayley compactifications of abelian groups},
  author={Develin, Mike},
  journal={Annals of Combinatorics},
  volume={6},
  number={3},
  pages={295--312},
  year={2002},
  publisher={Springer}
}

@article{RY23,
  title={Groups with finitely many {B}usemann points},
  author={Ron-George, Liran and Yadin, Ariel},
  journal={Groups, Geometry and Dynamics},
  year={2023},
  note={arXiv:2305.02303}
}

@article{RY24,
  title={Detecting virtual homomorphisms via {B}anach metrics},
  author={Ron-George, Liran and Yadin, Ariel},
  journal={arXiv preprint arXiv:2408.11543},
  year={2024}
}

@article{Walsh,
  title={The action of a nilpotent group on its horofunction boundary has finite orbits},
  author={Walsh, Cormac},
  journal={Groups, Geometry, and Dynamics},
  volume={5},
  number={1},
  pages={189--206},
  year={2011}
}

@article{karlsson2024metric,
  title={A metric fixed point theorem and some of its applications},
  author={Karlsson, Anders},
  journal={Geometric and Functional Analysis},
  pages={1--26},
  year={2024},
  publisher={Springer}
}

@article{karlsson2021linear,
  title={From linear to metric functional analysis},
  author={Karlsson, Anders},
  journal={Proceedings of the National Academy of Sciences},
  volume={118},
  number={28},
  pages={e2107069118},
  year={2021},
  publisher={National Acad Sciences}
}

@article{karlsson2021hahn,
  title={Hahn-{B}anach for metric functionals and horofunctions},
  author={Karlsson, Anders},
  journal={Journal of Functional Analysis},
  volume={281},
  number={2},
  pages={109030},
  year={2021},
  publisher={Elsevier}
}

@article{KB02,
  title={Boundaries of hyperbolic groups},
  author={Kapovich, Ilya and Benakli, Nadia},
  journal={arXiv preprint math/0202286},
  year={2002}
}

@inproceedings{gromov1981hyperbolic,
  title={Hyperbolic manifolds, groups and actions},
  author={Gromov, Mikhael},
  booktitle={Riemann surfaces and related topics: Proceedings of the 1978 Stony Brook Conference (State Univ. New York, Stony Brook, NY, 1978)},
  volume={97},
  pages={183--213},
  year={1981}
}

@article{karlsson2001non,
  title={Non-expanding maps and Busemann functions},
  author={Karlsson, Anders},
  journal={Ergodic Theory and Dynamical Systems},
  volume={21},
  number={5},
  pages={1447--1457},
  year={2001},
  publisher={Cambridge University Press}
}

@article{BT24,
  title={Horofunctions on the {H}eisenberg and {C}artan groups},
  author={Bodart, Corentin and Tashiro, Kenshiro},
  journal={arXiv preprint arXiv:2407.11943},
  year={2024}
}

@article{BF20,
  title={On the horofunction boundary of discrete {H}eisenberg group},
  author={Bader, Uri and Finkelshtein, Vladimir},
  journal={Geometriae Dedicata},
  volume={208},
  number={1},
  pages={113--127},
  year={2020},
  publisher={Springer}
}

@article{rieffel,
  title={Group ${C}^{*}$-algebras as compact quantum metric spaces},
  author={Rieffel, Marc A},
  journal={Documenta Mathematica},
  volume={7},
  pages={605--651},
  year={2002}
}

@article{WW06,
  title={{B}usemann points of infinite graphs},
  author={Webster, Corran and Winchester, Adam},
  journal={Transactions of the American Mathematical Society},
  volume={358},
  number={9},
  pages={4209--4224},
  year={2006}
}

@article{fisher2021sub,
  title={Sub-Finsler horofunction boundaries of the Heisenberg group},
  author={Fisher, Nate and Nicolussi Golo, Sebastiano},
  journal={Analysis and Geometry in Metric Spaces},
  volume={9},
  number={1},
  pages={19--52},
  year={2021},
  publisher={De Gruyter}}

@article{lamplighter,
  title={The Horofunction Boundary of the Lamplighter Group $L_2$ with the Diestel-Leader metric},
  author={Jones, Keith and Kelsey, Gregory A},
  journal={Topological Methods in Group Theory},
  volume={451},
  pages={111},
  year={2018},
  publisher={Cambridge University Press}}

\end{document}